\documentclass[dvipdfmx]{article}

\usepackage[margin=30truemm]{geometry}
\usepackage{indentfirst}
\usepackage[svgnames]{xcolor}
\usepackage{tikz}
\usepackage{amsmath,amssymb,amsthm,bm,ytableau,setspace}
\usepackage[enableskew]{youngtab} 

\DeclareMathOperator{\GL}{GL}
\DeclareMathOperator{\wt}{wt}
\DeclareMathOperator{\RR}{RR}

\ytableausetup{centertableaux}
\theoremstyle{definition} 
\newtheorem{theorem}{Theorem}[section]
\newtheorem{proposition}[theorem]{Proposition}
\newtheorem{definition}[theorem]{Defintion}
\newtheorem{remark}[theorem]{Remark}
\newtheorem{example}[theorem]{Example}

\newtheorem{lemma}[theorem]{Lemma}

\title{The Pieri formulas and the Littlewood-Richardson rule for Schur multiple zeta functions}

\author{Shutaro Nakaoka}
\begin{document}
\maketitle

\begin{abstract}
We prove the Pieri formulas for Schur multiple zeta functions, which are generalizations of the Pieri formulas proved by Nakasuji and Takeda for hook type Schur multiple zeta functions. Moreover, we also prove the Littlewood-Richardson rule for Schur multiple zeta functions.

In the course of their proofs, we regard the `truncated' version of Schur multiple zeta functions as series over $\GL(N)$ crystals to arrive at the Littlewood-Richardson rule for the Schur multiple zeta functions.

\end{abstract}

\section{Introduction}

Nakasuji-Phuksuwan-Yamasaki \cite{1} introduced the \textit{Schur multiple zeta functions} as a combinatorial generalization of multiple zeta functions. It is known that some Schur multiple zeta values have the properties inherited from the theory of multiple zeta values such as duality, shuffle product formula, and sum formula (see \cite{3,4,5}). On the other hand, Schur multiple zeta functions are variant of Schur functions, so one can expect similar properties to these of Schur functions. For example, Nakasuji-Phuksuwan-Yamasaki \cite{1} proved the Jacobi-Trudi formulas, the Giambelli formula, and the dual Cauchy formula for Schur multiple zeta functions.

The purpose of this paper is to generalize the Pieri formulas for hook type Schur multiple zeta functions proved by Nakasuji-Takeda \cite{2}. To that end, we first recall the definition of the Schur multiple zeta functions.

Let $X$ be a set and $\lambda$ be a partition. A Young tableau $T=(t_{ij})$ of shape $\lambda$ over $X$ is a filling of Young diagram obtained by putting $t_{ij}\in X$ into the $(i,j)$ box in the Young diagram of shape $\lambda$. Let $T(\lambda,X)$ denote the set of all Young tableaux of shape $\lambda$ over $X$. For a partition $\lambda=(\lambda_1,\ldots,\lambda_l)\;(\lambda_1\ge\lambda_2\ge\cdots\ge\lambda_l\ge0)$, the length of $\lambda$ is denoted by $l(\lambda)$, and the weight $|\lambda|$ of $\lambda$ is defined as $|\lambda|=\lambda_1+\cdots+\lambda_l$. For a partition $\lambda$, we define $D(\lambda)$ as follows:
\[
D(\lambda):=\{(i,j)\in \mathbb{Z}^2 \mid 1\leq i \leq l(\lambda), 1\leq j \leq \lambda_i \}.
\]
The conjugate of $\lambda$ is denoted by $\lambda'$.

Let $\lambda = (\lambda_1, \lambda_2, \dots, \lambda_k)$ and $\mu = (\mu_1, \mu_2, \dots, \mu_l)$ be two partitions such that $\mu_i \leq \lambda_i$ for all $i$. The skew partition $\lambda / \mu$ is defined as the set of boxes in the Young diagram of $\lambda$ that are not in the Young diagram of $\mu$. 

The weight $|\lambda/\mu|$ of a skew partition $\lambda/\mu$ is defined by
\[
|\lambda/\mu|=|\lambda|-|\mu|.
\]

A semistandard Young tableau of shape $\lambda/\mu$ is a filling of $\lambda/\mu$ with positive integers, subject to the following conditions:
\begin{enumerate}
    \item The entries in each row are weakly increasing (i.e., non-decreasing).
    \item The entries in each column are strictly increasing.
\end{enumerate}

\begin{definition}[Schur multiple zeta function, \cite{1}]
For a partition $\lambda$ and $\bm{s}=(s_{ij})\in T(\lambda,\mathbb{C})$, the Schur multiple zeta function is defined to be
\[
\zeta_{\lambda}(\bm{s})=\sum_{M\in SSYT(\lambda)}\frac{1}{M^{\bm{s}}}.
\]
where $SSYT(\lambda)\,( \subset\hspace{-3pt}T(\lambda, \mathbb{N}))$ denotes the set of all semistandard Young tableaux of shape $\lambda$, and $M^{\bm{s}}$ is defined as $M^{\bm{s}}=\prod_{(i,j)\in D(\lambda)}m_{ij}^{s_{ij}}$ for $M=(m_{ ij})\in SSYT(\lambda)$.
\end{definition}

Nakasuji and Takeda \cite{2} formulate the Pieri formulas for some special types of Schur multiple zeta functions by using `pushing rule' defined as follows:

\begin{definition}
Let $\lambda=(\lambda_1,\ldots,\lambda_l)$ be a partition. For a positive integer $n$, define $\mathcal{E}(\lambda,n)$ as the set of all subsets $K$ of $\mathbb{Z}_{>0}$ such that $\#K=n$ and the sequence $\lambda_K=((\lambda_K)_1,(\lambda_K)_2,(\lambda_K)_3,\ldots)$ given by
\[
(\lambda_K)_k=
\begin{cases}
 \lambda_k+1 & (k\in K) \\
 \lambda_k   & (k\notin K),
\end{cases}
\]
forms a partition. For a partition $\mu$ and a positive integer $m$, we define $\mathcal{H}(\mu,m)$ as $\mathcal{E}(\mu',m)$. For $J\in \mathcal{H}(\mu,m)$, the partition $((\mu')_J)'$ is denoted by $\mu^J$.

\end{definition}

\begin{definition}[{\cite[2.5]{2}}]
Let $\lambda$ be a partition and let $m$ be a positive integer. Let $J=\{\alpha_1,\ldots,\alpha_m\}\;(\alpha_1<\cdots<\alpha_m)$ be an element of $\mathcal{H}(\lambda,m)$. 
For $\bm{s}\in T(\lambda,\mathbb{C})$
and
$\bm{t}=
\begin{ytableau}
 t_1 & t_2 & \cdots & t_m
 \end{ytableau}
\in T((m),\mathbb{C})$, let $\bm{u}^J(\bm{s},\bm{t})$ denote the Young tableau over $\mathbb{C}$ of shape $\lambda^J$ defined by putting $t_j$ into the $(1,\alpha_j)$ box for $j\in J$ and putting $s_{ij}$ into the $(i,j)$ box if $j\notin J$ and $(i+1,j)$ box if $j\in J$ for $(i,j)\in D(\lambda)$.
\end{definition}

\begin{definition}[{\cite[2.5]{2}}]
Let $\lambda$ be a partition and let $n$ be a positive integer. Let $K=\{\beta_1,\cdots,\beta_n\}\;(\beta_1<\cdots<\beta_n)$ be an element of $\mathcal{E}(\lambda,n)$. For $
\bm{s}=
\begin{ytableau}
 s_1 \\
 s_2 \\
 \vdots \\
 s_n
\end{ytableau}
$
and
$\bm{t}\in T(\lambda,\mathbb{C})
$, let $\bm{u}_K(\bm{s},\bm{t})$ denote the Young tableau over $\mathbb{C}$ of shape $\lambda_K$ defined by putting $s_k$ into the $(\beta_k,1)$ box for $k\in K$ and putting $t_{ij}$ into the $(i,j)$ box if $i\notin K$ and into the $(i,j+1)$ box if $i\in K$ for $(i,j) \in D(\lambda)$.

\end{definition}

The following result is the main theorem of this paper, that can be understood as the Pieri formulas for Schur multiple zeta functions:

\begin{theorem}[= Theorem \ref{thm:i1}, Theorem \ref{thm:i2}]
\textit{Let $\lambda$ be a partition of length $p$ such that $\lambda_1=r$. Let $\lambda'=(\lambda_1',\ldots, \lambda_r')$ be its conjugate. Then, the following items hold:
\item[($H$-type)]
Let $\bm{s}\in T(\lambda,\mathbb{C})$
and
$\bm{t}=
\begin{ytableau}
 t_1 & t_2 & \cdots & t_m
\end{ytableau}
\in T((m),\mathbb{C})$. Suppose that the real parts of all $s_{ij},t_k$ are greater than or equal to $1$ and the real parts of 
\[
t_1,\ldots,t_{\min\{m,r-1\}},s_{11},\ldots,s_{{\lambda'}_2 1},s_{12},\ldots,s_{{\lambda'}_2 2},s_{13},\ldots,s_{{\lambda'}_3 3},\ldots,s_{1r},\ldots,s_{{\lambda'}_r r},t_m,s_{{\lambda'}_1 1}
\]
are greater than $1$. Then, the following equality holds:
\[\sum_{sym}\zeta_{\lambda}(\bm{s})\zeta_{(m)}(\bm{t})=\sum_{sym}\sum_{J\in \mathcal{H}(\lambda,m)} \zeta_{\lambda^J}(\bm{u}^J(\bm{s},\bm{t})).\]
Here $\displaystyle\sum_{sym}$ means the summation over the permutation of
\[
\{t_1,\ldots,t_{\min\{m,r-1\}},s_{11},\ldots,s_{{\lambda'}_2 1},s_{12},\ldots,s_{{\lambda'}_2 2},s_{13},\ldots,s_{{\lambda'}_3 3},\ldots,s_{1r},\ldots,s_{{\lambda'}_r r}\}
\]
as indeterminates.
\item[($E$-type)]
Let $\bm{s}=
\begin{ytableau}
 s_1 \\
 s_2 \\
 \vdots \\
 s_n
\end{ytableau}
$
and
$\bm{t}\in T(\lambda,\mathbb{C})
$. Suppose that the real parts of all $t_{ij},s_k$ are greater than or equal to $1$ and real parts of 
\[
s_1,\ldots,s_{\min\{n,p-1\}},t_{11},\ldots,t_{1\lambda_2},t_{21},\ldots,t_{2\lambda_2},t_{31},\ldots,t_{3\lambda_3},\ldots,t_{p1},\ldots,t_{p\lambda_p},s_n,t_{1\lambda_1}
\]
are greater than $1$. Then the following equality holds:
\[
\sum_{sym}\zeta_{(1^n)}(\bm{s})\zeta_{\lambda}(\bm{t})=\sum_{sym}\sum_{K\in \mathcal{E}(\lambda,n)}\zeta_{\lambda_K}(\bm{u}_K(\bm{s},\bm{t})).
\]
Here $\displaystyle\sum_{sym}$ means the summation over the permutation of
\[
\{s_1,\ldots,s_{\min\{n,p-1\}},t_{11},\ldots,t_{1\lambda_2 },t_{21},\ldots,t_{2\lambda_2},t_{31},\ldots,t_{3\lambda_3},\ldots,t_{p1},\ldots,t_{p\lambda_p}\}
\]
as indeterminates.
}
\end{theorem}

Also, we prove the following Littlewood-Richardson rule:

\begin{theorem}[= Theorem \ref{thm:lr}]
\textit{Let $\mu$ and $\nu$ be partitions and let $\bm{s}\in T(\mu,\mathbb{C})$, $\bm{t}\in T(\nu,\mathbb{C})$. Suppose that
the real parts of all $t_{ij}$ , $s_{kl}$ are greater than 1. For a skew partition $\lambda/\mu$ such that $|\lambda/\mu|=|\nu|$, we fix a filling $\bm{u}_{\lambda}(\bm{s},\bm{t})$ of $D(\lambda)$ with $\{s_{ij} \mid (i,j)\in D(\mu)\}\cup\{t_{kl} \mid (k,l)\in D(\nu)\}$ (as indeterminates). Then, the following holds:
\[
\sum_{sym}\zeta_{\mu}(\bm{s})\zeta_{\nu}(\bm{t})=\sum_{sym}\sum_{\lambda}c_{\mu\nu}^{\lambda}\zeta_{\lambda}(\bm{u}_{\lambda}(\bm{s},\bm{t})).
\]
Here $\displaystyle\sum_{sym}$ means the summation over the permutation of $\{s_{ij} \mid (i,j)\in D(\mu)\}\cup\{t_{kl} \mid (k,l)\in D(\nu)\}$ 
as indeterminates.
}

\end{theorem}

\section{Review on crystals}

We review on crystals in accordance with \cite{6}. We refer basics on partitions and Young diagram to \cite[Chapter 3]{6}

We fix a based root datum $(X^{*},R,\Pi,X_{*},R^{\vee},\Pi^{\vee})$, where $X^{*}$ is the weight lattice, $X_{*}$ is the coweight lattice, $R\subset X^{*}$ is the set of roots, $R^{\vee}\subset X_{*}$ is the set of coroots, $\Pi=\{\alpha_i\mid i\in I\}\subset R$ is the set of simple roots indexed by $I$, and $\Pi^{\vee}=\{\alpha_i^{\vee}\mid i\in I\}\subset R^{\vee}$ is the set of simple coroots. We denote by $\langle\cdot,\cdot\rangle : X^{*}\times X_{*}\to \mathbb{Z}$ the natural paring.

\begin{definition}[{\cite[Definition 2.13]{6}}]
A Kashiwara crystal is a nonempty set $\mathcal{B}$ together with maps
\[
e_i, f_i: \mathcal{B}\to \mathcal{B}\sqcup \{0\},
\epsilon_i, \phi_i: \mathcal{B}\to \mathbb{Z}\sqcup\{-\infty\}\;(i\in I),
\wt: \mathcal{B}\to X^{*},
\]
satisfying the following conditions:
\begin{enumerate}
\item[A1.] For $x,y \in \mathcal{B}$ and $i\in I$, we have $e_i(x)=y$ if and only if $f_i(y)=x$. In this case, we have
\[
\wt(y)=\wt(x)+\alpha_i, \epsilon_i(y)=\epsilon_i(x)-1, \phi_i(y)=\phi_i(x)+1.
\]
\item[A2.] We have
\[
\phi_i(x)=\langle \wt(x),\alpha_i^{\vee}\rangle+\epsilon_i(x)
\]
for all $x\in \mathcal{B}$ and $i\in I$. If $\phi_i(x)=-\infty$, then 
we have $e_i(x)=f_i(x)=0$.
\end{enumerate}
\end{definition}

If $\mathcal{B}$ is a crystal, we associate a directed graph with vertices consisting of all elements of $\mathcal{B}$ and edges labelled by $i\in I$. If $f_i(x)=y$, then we draw an edge labeled $i$ from $x$ to $y$. If we illustrate it in a graph, then it would look like following diagram:
\[
x\xrightarrow{i}y.
\]This is called the \textit{crystal graph} of $\mathcal{B}$. 

\begin{definition}[{\cite[2,2]{6}}]
A crystal $\mathcal{B}$ is called \textit{seminormal} if
\[
\phi_i(x)=\max\{k\in \mathbb{Z}_{\geq0} \mid f_i^k(x)\neq 0\}\;\text{and}\;\epsilon_i(x)=\max\{k\in \mathbb{Z}_{\geq0} \mid e_i^k\neq 0\}
\]
holds for all $i\in I$.
\end{definition}

\begin{example}[{\cite[Example 2.19]{6}}]
We assume that $I=\{1,\ldots,N-1\}$, $X^{*}=X_{*}=\mathbb{Z}^{N}$, $\langle\bm{e}_i,\bm{e}_j\rangle=\delta_{ij} \,(i,j\in I)$, and $\alpha_i=\alpha_i^{\vee}=\bm{e}_i-\bm{e}_{i+1}\, (i\in I)$ where $\bm{e}_i=(0,\ldots,1,\ldots,0)^T$ ($1$ in the $i$'th place).

There is a crystal whose crystal graph represented by the following diagram:
\[
\boxed{1}\xrightarrow{1}	\boxed{2}\xrightarrow{2} \cdots 
\xrightarrow{N-1}\boxed{N}\,.
\]
We define the weight by $\wt(\,\boxed{i}\,)=\bm{e}_i$. The maps $\phi_i$ and $\epsilon_i$ are uniquely determined if we require this crystal to be seminormal. We denote this crystal as $\mathbb{B}$.
\end{example}

\begin{definition}[{\cite[2.3]{6}}]
Two crystals $\mathcal{B}$ and $\mathcal{C}$ are called \textit{isomorphic} if there is a bijection $\psi:\mathcal{B}\sqcup\{0\} \to \mathcal{C}\sqcup\{0\}$ such that 
\begin{enumerate}
\item $\psi(0)=0$;
\item $\psi(e_i(b))=e_i(\psi(b))$ for all $b\in \mathcal{B}$ and $i\in I$;
\item $\psi(f_i(b))=f_i(\psi(b))$ for all $b\in \mathcal{B}$ and $i\in I$;
\item $\epsilon_i(\psi(b))=\epsilon_i(b)$ for all $b\in \mathcal{B}$ and $i\in I$;
\item $\phi_i(\psi(b))=\phi_i(b)$ for all $b\in \mathcal{B}$ and $i\in I$;
\item $\wt(\psi(b))=\wt(b)$ for all $b\in\mathcal{B}$.
\end{enumerate}
Note that in the situation described above, we have $\psi(b)\neq 0$ for $b\in \mathcal{B}$ since $\psi$ is an injection and $\psi(0)=0$.
\end{definition}

\begin{definition}[{\cite[2.3]{6}}]
Let $\mathcal{B}$ and $\mathcal{C}$ be crystals. We define the tensor product crystal $\mathcal{B}\otimes \mathcal{C}$ as follows:

As a set, $\mathcal{B}\otimes\mathcal{C}$ is $\mathcal{B}\times\mathcal{C}$, and we denote $(x,y)\in \mathcal{B}\times\mathcal{C}$ by $x\otimes y$. We define the weight by $\wt(x\otimes y)=\wt(x)+\wt(y)$ and the actions of Kashiwara operators $e_i, f_i:\mathcal{B}\otimes\mathcal{C}\to(\mathcal{B}\otimes\mathcal{C})\sqcup\{0\}\;(i\in I)$ by
\[
f_i(x\otimes y)=\left\{
\begin{array}{ll}
f_i(x)\otimes y & \mathrm{if}\; \phi_i(y) \leq \epsilon_i(x)\\
x\otimes f_i(y) & \mathrm{if} \;\phi_i(y)>\epsilon_i(x),
\end{array}
\right.
\]
\[
e_i(x\otimes y)=\left\{
\begin{array}{ll}
e_i(x)\otimes y & \mathrm{if}\; \phi_i(y)< \epsilon_i(x)\\
x\otimes e_i(y) & \mathrm{if}\; \phi_i(y) \geq \epsilon_i(x).
\end{array}
\right.
\]
We understand that $x\otimes 0=0\otimes x=0$. We also set 
\[
\phi_i(x\otimes y)=\max\{\phi_i(x), \phi_i(y)+\langle\wt(x), \alpha_i^{\vee}\rangle\}
\]
and
\[
\epsilon_i(x\otimes y)=\max\{\epsilon_i(y), \epsilon_i(x)-\langle\wt(y), \alpha_i^{\vee}\rangle\}.
\]
\end{definition}

\begin{proposition}[{\cite[Proposition 2.29]{6}}]
\textit{The tensor product $\mathcal{B}\otimes \mathcal{C}$ is a crystal. If $\mathcal{B}$ and $\mathcal{C}$ are seminormal, then so is $\mathcal{B}\otimes \mathcal{C}$. }
\end{proposition}

For a partition $\lambda=(\lambda_1,\ldots,\lambda_l)\;(\lambda_1\ge\lambda_2\ge\cdots\ge\lambda_l\ge0)$, the length of $\lambda$ is denoted by $l(\lambda)$ and the weight $|\lambda$ of $\lambda$ is defined as $|\lambda|=\lambda_1+\cdots+\lambda_l$.

Let $\lambda = (\lambda_1, \lambda_2, \dots, \lambda_k)$ and $\mu = (\mu_1, \mu_2, \dots, \mu_l)$ be two partitions such that $\mu_i \leq \lambda_i$ for all $i$. The skew partition $\lambda / \mu$ is the set of boxes in the Young diagram of $\lambda$ that are not in the Young diagram of $\mu$. This can be expressed as:

\[
\lambda / \mu = \{ (i,j) \mid 1 \leq i \leq k, \, \mu_i < j \leq \lambda_i \}.
\]

The weight of the skew partition $\lambda/\mu$ is given by the difference of the weights of the two partitions:
\[
|\lambda/\mu|=|\lambda|-|\mu|.
\]

A semistandard Young tableau of shape $\lambda/\mu$ is a filling of $\lambda/\mu$ with positive integers, subject to the following conditions:
\begin{enumerate}
    \item The entries in each row are weakly increasing (i.e., non-decreasing).
    \item The entries in each column are strictly increasing.
\end{enumerate}

Let $\lambda$ be a partition such that $l(\lambda)\leq N$. We write $\mathcal{B}_{\lambda}$ to be the set of all semistandard tableaux of shape $\lambda$ in the alphabet $\{1,\ldots,N\}$. We also denote it by $\mathcal{B}_{\lambda}^{(N)}$. If $T\in \mathcal{B}_{\lambda}$, we define $\RR(T)\in \mathbb{B}^{\otimes |\lambda|}$ by reading each row of $T$ in order, and we take the rows in order from the bottom to the top. For example, if
\[T=
\begin{ytableau}
 1  & 1 & 2 \\
 2  & 3 \\
 4 \\
\end{ytableau}\;,
\]
then
\[
\RR(T)=\boxed{4}\otimes\boxed{2}\otimes\boxed{3}\otimes\boxed{1}\otimes\boxed{1}\otimes\boxed{2}\,.
\]

Then, we get a map $\RR: \mathcal{B}_{\lambda}\to \mathbb{B}^{\otimes |\lambda|}$.

\begin{theorem}[{\cite[Theorem 3.2]{6}}]
\textit{Let $\lambda$ be a partition of $k$ of length $\leq N$. Then, $\RR(\mathcal{B}_{\lambda})$ is a connected component of $\mathbb{B}^{\otimes k}$. }
\end{theorem}

The above theorem enable us to define the crystal structure on $\mathcal{B}_{\lambda}$ by identifying it with $\RR(\mathcal{B}_{\lambda})$.

\begin{definition}[reading word]
For a skew semistandard tableau, its reading word is the word obtained by concatenating the rows, starting from the bottom row. 
\end{definition}

\begin{example}
If $T$ is a skew tableau
\[
\young(:1123,223,3)\,,
\]
then the reading word is $32231123$.
\end{example}

\begin{definition}[{\cite[8.4]{6}}]
A word $u_1\cdots u_k$  is called a \textit{Yamanouchi word} if for all $i$ and $j$,  the number of $i$'s appearing in $\{u_j, \ldots, u_k\}$ is greater than or equal to the number of $(i+1)$'s.
\end{definition}

\begin{definition}[{\cite[8.2]{6}}]
Let $\lambda/\mu$ be a skew partition and let $\nu=(\nu_1,\nu_2,\nu_3,\ldots)$ be a partition. The \textit{Littlewood-Richardson coefficient} $c_{\mu\nu}^{\lambda}$ is defined as the number of skew semistandard tableaux $T$ of shape $\lambda/\mu$ such that the reading word of $T$ is a Yamanouchi word, and the entry $j$ appears exactly $\nu_j$ times for each $j\in \mathbb{Z}_{>0}$.
\end{definition}

For a semistandard tableau $T$ and a positive integer $w$, we define the Schensted row insertion $T \gets w$ as in \cite[Section 1.1]{7}. Also, for a semistandard tableau $T$ and a positive integer $u$, we define the Schensted column insertion $u \to T$ as in \cite[Appendix A]{7}.

  \begin{theorem}[{\cite[Theorem 8.6, Theorem 9.5]{6}}]\label{thm:lrc}
\textit{Let $\mu$ and $\nu$ be partitions such that $l(\mu), \l(\nu)\leq N$. Then 
\[
\mathcal{B}_{\mu}^{(N)}\otimes\mathcal{B}_{\nu}^{(N)}
\cong \bigoplus_{l(\lambda)\leq N} (\mathcal{B}_{\lambda}^{(N)})^{\oplus c_{\mu\nu}^{\lambda}}.
\]
Moreover, $L\otimes M\in \mathcal{B}_{\mu}^{(N)}\otimes\mathcal{B}_{\nu}^{(N)}$ corresponds to the result of applying the sequence of Schensted row insertions
\[
(\cdots(L \gets w_1 ) \gets \cdots ) \gets w_{|\nu|}.
\]
Here $w_1\cdots w_{|\nu|}$ is the reading word of $M$.}
\end{theorem}

\begin{remark}
In the setting of the above theorem, we let $u_1\cdots u_{|\mu|}$ be the reading word of $L$. By \cite[A.2]{7}, we have
\begin{align*}
(\cdots(L \gets w_1 ) \gets \cdots ) \gets w_{|\nu|}&=(\cdots((\cdots(\varnothing \gets u_1) \gets \cdots) \gets u_{|\mu|}) \gets w_1 \gets \cdots) \gets w_{|\nu|} \\
&=u_1 \to (\cdots \to (u_{|\mu|} \to (w_1 \to (\cdots \to (w_{|\nu|} \to \varnothing)\cdots)))\cdots) \\
&=u_1 \to (\cdots \to (u_{|\mu|} \to M)\cdots).
\end{align*}
\end{remark}

\section{The Schur multiple zeta functions}

Keep the setting of the previous section.

The combinatorial setting described below is conformed to \cite{1}.

For a set $X$ and a partition $\lambda$, the Young tableau $T=(t_{ij})$ of shape $\lambda$ over $X$ is a filling of Young diagram obtained by putting $t_{ij}\in X$ into the $(i,j)$ box in the Young diagram of shape $\lambda$. Let $T(\lambda, X)$ denote the set of all Young tableau of shape $\lambda$ over $X$. For a partition $\lambda$, we define 
\[
D(\lambda):=\{(i,j)\in \mathbb{Z}^2 \mid 1\leq i \leq l(\lambda), 1\leq j \leq \lambda_i \}
\] and the conjugate of $\lambda$ is denoted by $\lambda'$.

For a partition $\lambda$, $(i,j)\in D(\lambda)$ is called a corner of $\lambda$ if $(i+1,j)\notin D(\lambda)$ and $(i,j+1)\notin D(\lambda)$. We denote the set of all corners of $\lambda$ by $C(\lambda)$.

\begin{definition}[Schur multiple zeta function, \cite{1}]
For a partition $\lambda$ and $\bm{s}=(s_{ij}) \in T(\lambda,\mathbb{C})$, the Schur multiple zeta function is defined to be
\begin{equation}
\label{dd}
\zeta_{\lambda}(\bm{s})=\sum_{M\in SSYT(\lambda)}\frac{1}{M^{\bm{s}}}.
\end{equation}
Here $SSYT(\lambda)\,( \subset\hspace{-3pt}T(\lambda, \mathbb{N}))$ is the set of all semistandard Young tableau of shape $\lambda$ and $M^{\bm{s}}=\prod_{(i,j)\in D(\lambda)}m_{ij}^{s_{ij}}$ for $M=(m_{ ij})\in SSYT(\lambda)$.
\end{definition}

\begin{remark}
If $\bm{s}\in T(\lambda,\mathbb{R})$, then each term of the series (\ref{dd}) is positive. Hence, if the series (\ref{dd}) is convergent, then it is absolutely convergent. Thus, any rearrangement converges to the same value. 
\end{remark}

\begin{proposition}[{\cite[Lemma 2.1]{1}}]\label{thm:c1}
\textit{The series (\ref{dd}) is absolutely convergent in
\[
W_{\lambda}=\{\bm{s}=(s_{ij}) \mid \mathrm{Re}(s_{ij})\geq 1 \text{ for all } (i,j)\in D(\lambda), \mathrm{Re}(s_{ij})>1 \text{ for all } (i,j)\in C(\lambda)\}. 
\]}
\end{proposition}

There is the `truncated' version of Schur multiple functions.

\begin{definition}[\cite{1}]
Let $\lambda$ be a partition such that $l(\lambda)\leq N$ and let $\bm{s}\in T(\lambda,\mathbb{C})$. Then the function $\zeta_{\lambda}^{(N)}(\bm{s})$ is defined to be
\[
\zeta_{\lambda}^{(N)}(\bm{s})=\sum_{M\in\mathcal{B}_{\lambda}^{(N)}}\frac{1}{M^{\bm{s}}}.
\]
\end{definition}

\begin{remark}\label{thm:ww}
We have $\mathcal{B}_{\lambda}^{(N)}\subset \mathcal{B}_{\lambda}^{(N')}\subset SSYT(\lambda)\;(N\le N')$ and $\bigcup_{N}\mathcal{B}_{\lambda}^{(N)}=SSYT(\lambda)$. In particular, we have 
\[
\lim_{N\to \infty}\sum_{M\in \mathcal{B}_{\lambda}^{(N)}}\frac{1}{M^{\bm{s}}}=\sum_{M\in SSYT(\lambda)}\frac{1}{M^{\bm{s}}}
\]
for $\bm{s}\in W_{\lambda}$.
\end{remark}

\begin{remark}
The \textit{multiple zeta functions} are defined to be
\[
\zeta(s_1,\cdots,s_d)=\sum_{0<m_1<\cdots<m_d}\frac{1}{m_1^{s_1}\cdots m_d^{s_d}}
\]
and the \textit{multiple zeta-star functions} are defined to be
\[
\zeta^{\star}(s_1,\cdots,s_d)\sum_{0<m_1\leq\cdots\leq m_d}\frac{1}{m_1^{s_1}\cdots m_d^{s_d}}.
\]
One can see  that the multiple zeta function is the Schur multiple zeta function of type $(d)$ and the multiple zeta-star function is the Schur multiple zeta function of type $(1^d)$.
\end{remark}

\section{The Pieri formulas}

Keep the setting of the previous section.

\subsection{The pushing rule}

To formulate the Pieri formulas for Schur multiple zeta functions, we recall the `pushing rule' for tableaux introduced by Nakasuji-Takeda \cite{2}.

\begin{definition}
Let $\lambda=(\lambda_1,\ldots,\lambda_l)$ be a partition. For a positive integer $n$, let $\mathcal{E}(\lambda,n)$ denote the set of all subsets $K$ of $\mathbb{Z}_{>0}$ satisfying the following conditions:
\begin{enumerate}
  \item $\#K=n$.
  \item The sequence $\lambda_K=((\lambda_K)_1,(\lambda_K)_2,(\lambda_K)_3,\ldots)$ defined by
\[
(\lambda_K)_k=
\begin{cases}
 \lambda_k+1 & (k\in K) \\
 \lambda_k   & (k\notin K),
\end{cases}
\]
forms a partition.
\end{enumerate}

 For a partition $\mu$ and a positive integer $m$, we define $\mathcal{H}(\mu,m)$ to be $\mathcal{E}(\mu',m)$. For $J\in \mathcal{H}(\mu,m)$, the partition $((\mu')_J)'$ is denoted by $\mu^J$.

\end{definition}

\begin{definition}[{\cite[2.5]{2}}]\label{thm:gak2}
Let $\lambda$ be a partition and let $m$ be a positive integer. Let $J=\{\alpha_1,\ldots,\alpha_m\}\;(\alpha_1<\cdots<\alpha_m)$ be an element of $\mathcal{H}(\lambda,m)$. 
For $\bm{s}\in T(\lambda,\mathbb{C})$
and
$\bm{t}=
\begin{ytableau}
 t_1 & t_2 & \cdots & t_m
 \end{ytableau}
\in T((m),\mathbb{C})$, the Young tableau $\bm{u}^J(\bm{s},\bm{t})$ over $\mathbb{C}$ of shape $\lambda^J$ is defined as follows:
\begin{enumerate}
  \item For each $j$, place $t_j$ in the box $(1,\alpha_j)$.
  \item For $(i,j)\in D(\lambda)$,
  \begin{enumerate}
    \item place $s_{ij}$ in the box $(i,j)$ if $j\notin J$,
    \item place $s_{ij}$ in the box $(i+1,j)$ if $j\in J$.
  \end{enumerate}
\end{enumerate}

\end{definition}

\begin{example}
In Definition \ref{thm:gak2}, we fix $\lambda=(3,2,1,1), m=3$ and $J=\{1,3,4\}$. Then, $\bm{u}^J(\bm{s},\bm{t})$ is represented by the following diagram:
\[
\begin{ytableau}
 t_1 & s_{12} & t_2 & t_3  \\
 s_{11} & s_{22} & s_{13}  \\
 s_{21}   \\
 s_{31}\\
 s_{41}
\end{ytableau}\;.
\]
\end{example}

\begin{definition}[{\cite[2.5]{2}}]\label{thm:gak1}
Let $\lambda$ be a partition and let $n$ be a positive integer. Let $K=\{\beta_1,\cdots,\beta_n\}\;(\beta_1<\cdots<\beta_n)$ be an element of $\mathcal{E}(\lambda,n)$. For $
\bm{s}=
\begin{ytableau}
 s_1 \\
 s_2 \\
 \vdots \\
 s_n
\end{ytableau}
$
and
$\bm{t}\in T(\lambda,\mathbb{C})
$, the Young tableau $\bm{u}_K(\bm{s},\bm{t})$ over $\mathbb{C}$ of shape $\lambda_K$ is defined as follows:
\begin{enumerate}
  \item For each $k\in K$, place $s_k$ in the box $(\beta_k,1)$.
  \item For $(i,j)\in D(\lambda)$,
  \begin{enumerate}
    \item place $t_{ij}$ in the box $(i,j)$ if $i\notin K$,
    \item place $t_{ij}$ in the box $(i,j+1)$ if $i\in K$.
  \end{enumerate}
\end{enumerate}

\end{definition}

\begin{example}
In Definition \ref{thm:gak1}, we fix $\lambda=(3,2,1,1), n=4$ and $K=\{1,3,5,6\}$. Then, $\bm{u}_K(\bm{s},\bm{t})$ is represented by the following diagram:
\[
\begin{ytableau}
 s_1 & t_{11} & t_{12} & t_{13}  \\
 t_{21} & t_{22}  \\
 s_2 & t_{31} \\
 t_{41}\\
 s_3 \\
 s_4 \\
\end{ytableau}\;.
\]
\end{example}

\subsection{The Pieri formula for $H$-type}

\begin{theorem}\label{thm:i1}
\textit{Let $\lambda$ be a partition with $\lambda_1=r$, and let $\lambda'=(\lambda_1',\ldots, \lambda_r')$ denote its conjugate. Let $\bm{s}\in T(\lambda,\mathbb{C})$
and
$\bm{t}=
\begin{ytableau}
 t_1 & t_2 & \cdots & t_m
\end{ytableau}
\in T((m),\mathbb{C})$. Assume that the real parts of all $s_{ij},t_k$ are greater than or equal to $1$ and the real parts of 
\[
t_1,\ldots,t_{\min\{m,r-1\}},s_{11},\ldots,s_{{\lambda'}_2 1},s_{12},\ldots,s_{{\lambda'}_2 2},s_{13},\ldots,s_{{\lambda'}_3 3},\ldots,s_{1r},\ldots,s_{{\lambda'}_r r},t_m,s_{{\lambda'}_1 1}
\]
are greater than $1$. Then, the following identity holds:
\[\sum_{sym}\zeta_{\lambda}(\bm{s})\zeta_{(m)}(\bm{t})=\sum_{sym}\sum_{J\in \mathcal{H}(\lambda,m)} \zeta_{\lambda^J}(\bm{u}^J(\bm{s},\bm{t})).\]
Here $\displaystyle\sum_{sym}$ denotes the summation over all permutations of the indeterminates
\[
\{t_1,\ldots,t_{\min\{m,r-1\}},s_{11},\ldots,s_{{\lambda'}_2 1},s_{12},\ldots,s_{{\lambda'}_2 2},s_{13},\ldots,s_{{\lambda'}_3 3},\ldots,s_{1r},\ldots,s_{{\lambda'}_r r}\}.
\]
}
\end{theorem}

To prove Theorem \ref{thm:i1}, we need the following lemma:

\begin{lemma}\label{thm:asa}
\textit{We keep the setting of the above theorem. We assume that $N\geq l(\lambda)+1$. Let $L\in \mathcal{B}_{\lambda}^{(N)}$ and 
\[
M=\begin{ytableau}
w_1 & \cdots & w_m
\end{ytableau} \in \mathcal{B}_{(m)}^{(N)}.
\]
Further, suppose that the tableau
\[
T=(\cdots(L \gets w_1) \gets \cdots ) \gets w_m
\]
has shape $\lambda^J$, where $J=\{\alpha_1,\ldots,\alpha_m\}\;(\alpha_1<\cdots<\alpha_m)$ is an element of $\mathcal{H}(\lambda,m)$. Then the following identity holds:
\[\sum_{sym} \frac{1}{L^{\bm{s}}M^{\bm{t}}}=\sum_{sym}\frac{1}{T^{\bm{u}^J(\bm{s},\bm{t})}}.
\]}
\end{lemma}

\begin{proof}
Assume that during the insertion $w_l$, the elements $x_{l,1}, x_{l,2}, x_{l,3}, \ldots, x_{l,a_l}$ are bumped from the boxes $(1,j_{l,1}), (2,j_{l,2}), (3,j_{l,3}), \ldots, (a_l, j_{l,a_l})$ respectively, and $x_{l,a_l}$ is placed in the box $(a_l+1,j_{l,a_l+1})$. If $w_l$ is directly placed at the end of the first row, we set $a_l=0$ and define $x_{l,0}=w_l$. The set of the boxes $(1,j_{l,1}), (2,j_{l,2}), (3,j_{l,3}), \ldots, (a_l, j_{l,a_l}), (a_l+1,j_{l,a_l+1})$ is called the \textit{bumping route} and denoted by $R_l$. 

By \cite[1.1]{7}, we have the following inequalities for all $l=1,\cdots,m$:
\[
  j_{l,1}\geq j_{l,2}\geq j_{l,3}\geq \cdots \geq j_{l,a_l+1}\label{j1}.
\]
In addition, the bumping route $R_{l-1}$ is strictly left to the bumping route $R_l$ for $l=2,\ldots,m$. That is, we have
\[
  j_{l-1,\eta}<j_{l,\eta}\label{j2}
\]
for all $\eta=1,\ldots, a_l+1$ by \cite[1.1]{7}. During the insertion process, the boxes $(a_1+1,j_{1,a_1+1}),\ldots,(a_m+1,j_{m,a_m+1})$ are added to the partition $\lambda$. Hence we have $j_{l,a_l+1}=\alpha_l$ for $l=1,\ldots,m$.

Let $\bm{v}$ denote the tableau over $\mathbb{C}$ of shape $\lambda^J$ constructed as follows:
\begin{enumerate}
  \item Place $t_l$ in the box $(1,j_{l,1})$ for $l=1,\ldots,m$.
  \item Place $s_{\eta-1,j_{l,\eta-1}}$ in the box $(\eta,j_{l,\eta})$ for $l=1,\ldots,m$ and $\eta=2,\ldots, a_l+1$.
  \item Place $s_{ij}$ in the box $(i,j)$ for $(i,j)\in D(\lambda)-(R_1\cup\cdots\cup R_m)$.
\end{enumerate}

From the construction of $\bm{v}$, we have $\displaystyle\frac{1}{L^{\bm{s}}M^{\bm{t}}}=\frac{1}{T^{\bm{v}}}$.

Next, we compare $\bm{v}$ and $\bm{u}^J(\bm{s},\bm{t})$. The shapes of the tableaux $\bm{v}$ and $\bm{u}^J(\bm{s},\bm{t})$ are identical. Let $S_l$ denote the set of the boxes $(1,\alpha_l), (2,\alpha_l), \ldots, (a_{l}+1,\alpha_l)$. If we write $S_l\cap R_l=\{(\eta_{l},\alpha_l),(\eta_l+1,\alpha_l),\ldots, (a_l+1,\alpha_l)\}$, we see that the entries of the $(i,\alpha_l)$ boxes in $\bm{v}$ and $\bm{u}^J(\bm{s},\bm{t})$ are equal to $s_{i-1,\alpha_l}$ for $i=\eta_l+1,\eta_l+2,\ldots,a_l+1$ by the construction of $\bm{v}$ and $\bm{u}^J(\bm{s},\bm{t})$. In particular, if $\alpha_l=1$, the entries of the $(i,1)$ boxes in $\bm{v}$ and $\bm{u}^J(\bm{s},\bm{t})$ are equal to $s_{(i-1)1}$ for $i=\lambda_2'+2,\lambda_2'+3,\ldots,\lambda_1'+1$. If $\alpha_1\geq 2$, the entries of the $(i,1)$ boxes in $\bm{v}$ and $\bm{u}^J(\bm{s},\bm{t})$ are equal to $s_{i1}$ for $i=1,2,\ldots,\lambda_1'$.

 Let $l_0$ be the largest index such that $a_l>0$. Then, $l_0\leq r$ and the entries of the $(1,j)$ boxes in $\bm{v}$ and $\bm{u}^J(\bm{s},\bm{t})$ are equal to $t_{j-r+l_0}$ for $j=r+1,\ldots,m+r-l_0$.

 If $l_0=r$, then $j_{l,\eta}=l$ for all $l=1,\ldots,r$ and $\eta=1,\ldots, a_l+1$. Hence the entries of the $(1,r)$ boxes in $\bm{v}$ and $\bm{u}^J(\bm{s},\bm{t})$ are equal to $t_r$. Thus, $\bm{v}$ and $\bm{u}^J(\bm{s},\bm{t})$ are equal up to the permutation of
\[
\{t_1,\ldots,t_{\min\{m,r-1\}},s_{11},\ldots,s_{{\lambda'}_2 1},s_{12},\ldots,s_{{\lambda'}_2 2},s_{13},\ldots,s_{{\lambda'}_3 3},\ldots,s_{1r},\ldots,s_{{\lambda'}_r r}\}.
\]

Therefore, the equality
\[
\sum_{sym} \frac{1}{L^{\bm{s}}M^{\bm{t}}}=\sum_{sym}\frac{1}{T^{\bm{v}}}=\sum_{sym}\frac{1}{T^{\bm{u}^J(\bm{s},\bm{t})}}
\]
holds.
\end{proof}

\begin{proof}[Proof of Theorem \ref{thm:i1}]
Let $N$ be an integer such that $N\geq l(\lambda)+1$. We have
\begin{align*}
\sum_{sym}\zeta_{\lambda}^{(N)}(\bm{s})\zeta_{(m)}^{(N)}(\bm{t})&=\sum_{sym}\sum_{L\otimes M\in \mathcal{B}_{\lambda}^{(N)}\otimes\mathcal{B}_{(m)}^{(N)}}\frac{1}{L^{\bm{s}}M^{\bm{t}}}\\
&=\sum_{sym}\sum_{J\in \mathcal{H}(\lambda,m)}\sum_{T\in \mathcal{B}_{\lambda^J}^{(N)}}\frac{1}{T^{\bm{u}^J(\bm{s},\bm{t})}} & \text{($\because$ Theorem \ref{thm:lrc} and Lemma \ref{thm:asa})} \\
&=\sum_{sym}\sum_{J\in\mathcal{H}(\lambda,m)} \zeta_{\lambda^J}^{(N)}(\bm{u}^J(\bm{s},\bm{t})).
\end{align*}

By Proposition \ref{thm:c1}, the sums $\displaystyle \sum_{L\in SSYT(\lambda)}\frac{1}{L^{\bm{s}}},\sum_{M\in SSYT((m))}\frac{1}{M^{\bm{t}}}$, and $\displaystyle\sum_{T\in SSYT(\lambda^J)}\frac{1}{T^{\bm{u}^J(\bm{s},\bm{t})}}$ are absolutely convergent. Hence taking the limit $N\to \infty$, we get the desired equality by Remark \ref{thm:ww}.
\end{proof}

\subsection{The Pieri formula for $E$-type}

In this subsection, we consider the summation of the products $\zeta_{\lambda}\cdot\zeta_{(1^n)}$ with swapped variables. The discussion in this subsection is parallel to that of previous subsection.

\begin{theorem}\label{thm:i2}
\textit{Let $\lambda$ be a partition of length $p$ and let $\bm{s}=
\begin{ytableau}
 s_1 \\
 s_2 \\
 \vdots \\
 s_n
\end{ytableau}
$
and
$\bm{t}\in T(\lambda,\mathbb{C})
$. Suppose that the real parts of all $t_{ij},s_k$ are greater than or equal to $1$ and real parts of 
\[
s_1,\ldots,s_{\min\{n,p-1\}},t_{11},\ldots,t_{1\lambda_2},t_{21},\ldots,t_{2\lambda_2},t_{31},\ldots,t_{3\lambda_3},\ldots,t_{p1},\ldots,t_{p\lambda_p},s_n,t_{1\lambda_1}
\]
are greater than $1$. Then, the following equality holds:
\[
\sum_{sym}\zeta_{(1^n)}(\bm{s})\zeta_{\lambda}(\bm{t})=\sum_{sym}\sum_{K\in \mathcal{E}(\lambda,n)}\zeta_{\lambda_K}(\bm{u}_K(\bm{s},\bm{t})).
\]
Here $\displaystyle\sum_{sym}$ means the summation over the permutation of
\[
\{s_1,\ldots,s_{\min\{n,p-1\}},t_{11},\ldots,t_{1\lambda_2 },t_{21},\ldots,t_{2\lambda_2},t_{31},\ldots,t_{3\lambda_3},\ldots,t_{p1},\ldots,t_{p\lambda_p}\}
\]
as indeterminates.
}
\end{theorem}

To prove Theorem \ref{thm:i2}, we need the following lemma:

\begin{lemma}\label{thm:hru}
\textit{We keep the setting of the above theorem. We assume that $N\geq l(\lambda)+n$. Let $L=\begin{ytableau}
u_1 \\
\vdots \\
u_n
\end{ytableau}
\in \mathcal{B}_{(1^n)}^{(N)}$ and
$M\in \mathcal{B}_{\lambda}^{(N)}$. We assume that the shape of the tableau 
\[
u_n\to (\cdots \to (u_1 \to M)\cdots)
\]
is $\lambda_K$, where $K=\{\beta_1,\ldots,\beta_n\}\;(\beta_1<\cdots<\beta_n)$ is an element of $\mathcal{E}(\lambda,n)$. Then 
\[\sum_{sym} \frac{1}{L^{\bm{s}}M^{\bm{t}}}=\sum_{sym}\frac{1}{T^{\bm{u}_K(\bm{s},\bm{t})}}.
\]}
\end{lemma}

\begin{proof}
The proof is similar to that of Lemma \ref{thm:asa}.

Suppose that during the insertion $u_l$, the elements $y_{l,1}, y_{l,2}, y_{l,3}, \ldots, y_{l,b_l}$ are bumped from boxes $(i_{l,1},1), (i_{l,2},2), (i_{l,3},3), \ldots, (i_{l,b_l},b_l)$ respectively, and $y_{l,b_l}$ is placed at $(i_{l,b_l+1},b_l+1)$. Here we understand that $b_l=0$ if $u_l$ is placed at the end of the first column and we set $y_{l,0}=u_l$. The set of the boxes $(i_{l,1},1), (i_{l,2},2), (i_{l,3},3), \ldots, (i_{l,b_l},b_l), (i_{l,b_l+1},b_l+1)$ is called the \textit{bumping route} and denoted by $R_l$.

We have
\[
i_{l,1}\geq i_{l,2}\geq i_{l,3}\geq \cdots \geq i_{l,b_l+1}
\]
for all $l=1,\cdots,n$ by \cite[A.2]{7}. In addition, the bumping route $R_{l-1}$ is strictly above the bumping route $R_l$ for $l=2,\ldots,n$, that is we have 
\[
i_{l-1,\eta}<i_{l,\eta}
\]
for $\eta=1,\ldots, b_l+1$ by \cite[A.2]{7}. During the insertion, the boxes $(i_{1,b_1+1},b_1+1),\ldots,(i_{n,b_n+1},b_n+1)$ are added to $\lambda$. Hence we have $i_{l,b_l+1}=\beta_l$ for $l=1,\ldots,n$.

Let $\bm{w}$ denote the tableau over $\mathbb{C}$ of shape $\lambda_K$ defined by putting $s_l$ into the $(i_{l,1},1)$ box for $l=1,\ldots,n$, putting $t_{i_{l,\eta-1},\eta-1}$ into the $(i_{l,\eta},\eta)$ box for $l=1,\ldots,n$ and $\eta=2,\ldots, b_l+1$, and putting $t_{ij}$ into the $(i,j)$ box for $D(\lambda)-(R_1\cup\cdots\cup R_n)$. From the construction of $\bm{w}$, we have $\displaystyle\frac{1}{L^{\bm{s}}M^{\bm{t}}}=\frac{1}{T^{\bm{w}}}$.

We compare $\bm{w}$ and $\bm{u}_K(\bm{s},\bm{t})$. The shape of the tableau $\bm{w}$ is equal to that of the tableau $\bm{u}_K(\bm{s},\bm{t})$. Let $S_l$ denote the set of the boxes $(\beta_l,1), (\beta_l,2), \ldots, (\beta_l,b_l+1)$. If we write $S_l\cap R_l=\{(\beta_l,\eta_l),(\beta_l,\eta_l+1),\ldots, (\beta_l,b_l+1)\}$, then the entries of the $(\beta_l,j)$ boxes in $\bm{v}$ and $\bm{u}_K(\bm{s},\bm{t})$ are equal to $t_{\beta_l,j-1}$ for $j=\eta_l+1,\eta_l+2,\ldots,b_l+1$ by the construction of $\bm{w}$ and $\bm{u}_K(\bm{s},\bm{t})$. In particular, if $\beta_1=1$, then the entries of the $(1,j)$ boxes in $\bm{w}$ and $\bm{u}_K(\bm{s},\bm{t})$ are equal to $t_{1,j-1}$ for $j=\lambda_2+2,\lambda_2+3,\ldots,\lambda_1+1$. If $\beta_1\geq 2$, then the entries of the $(1,j)$ boxes in $\bm{w}$ and $\bm{u}_K(\bm{s},\bm{t})$ are equal to $t_{1j}$ for $j=1,2,\ldots,\lambda_1$.

We write the largest $l$ such that $b_l>0$ by $l_0$. Then, $l_0\leq p$ and the entries of the $(i,1)$ boxes in $\bm{w}$ and $\bm{u}_K(\bm{s},\bm{t})$ are equal to $s_{i-p+l_0}$ for $i=p+1,\ldots,n+p-l_0$.

If $l_0=p$, then $i_{l,\eta}=l$ for all $l=1,\ldots,p$ and $\eta=1,\ldots, b_l+1$. Hence the entries of the $(p,1)$ boxes in $\bm{w}$ and $\bm{u}_K(\bm{s},\bm{t})$ are equal to $s_p$. Thus, $\bm{w}$ and $\bm{u}_K(\bm{s},\bm{t})$ are equal up to the permutation of

\[
\{s_1,\ldots,s_{\min\{n,p-1\}},t_{11},\ldots,t_{1\lambda_2 },t_{21},\ldots,t_{2\lambda_2},t_{31},\ldots,t_{3\lambda_3},\ldots,t_{p1},\ldots,t_{p\lambda_p}\}.
\]

Therefore, the equality
\[
\sum_{sym} \frac{1}{L^{\bm{s}}M^{\bm{t}}}=\sum_{sym}\frac{1}{T^{\bm{w}}}
=\sum_{sym}\frac{1}{T^{\bm{u}_K(\bm{s},\bm{t})}}
\]
holds.
\end{proof}

\begin{proof}[Proof of Theorem \ref{thm:i2}]
Let $N$ be an integer such that $N\geq l(\lambda)+n$. We have
\begin{align*}
\sum_{sym}\zeta_{(1^n)}^{(N)}(\bm{s})\zeta_{\lambda}^{(N)}(\bm{t})&=\sum_{sym}\sum_{L\otimes M\in \mathcal{B}_{(1^n)}^{(N)}\otimes\mathcal{B}_{\lambda}^{(N)}}\frac{1}{L^{\bm{s}}M^{\bm{t}}}\\
&=\sum_{sym}\sum_{K\in \mathcal{E}(\lambda,n)}\sum_{T\in \mathcal{B}_{\lambda_K}^{(N)}}\frac{1}{T^{\bm{u}_K(\bm{s},\bm{t})}} & \text{($\because$ Theorem \ref{thm:lrc} and Lemma \ref{thm:hru})}\\
&=\sum_{sym}\sum_{K\in\mathcal{E}(\lambda,n)} \zeta_{\lambda_K}^{(N)}(\bm{u}_K(\bm{s},\bm{t})).
\end{align*}

By Proposition \ref{thm:c1}, the sums $\displaystyle \sum_{L\in SSYT((1)^n)}\frac{1}{L^{\bm{s}}},\sum_{M\in SSYT(\lambda)}\frac{1}{M^{\bm{t}}}$ and $\displaystyle\sum_{T\in SSYT(\lambda_K)}\frac{1}{T^{\bm{u}_K(\bm{s},\bm{t})}}$ are absolutely convergent. Taking the limit $N\to \infty$, we get the desired equality by Remark \ref{thm:ww}.
\end{proof}

\section{Littlewood-Richardson rule}

Keep the setting of the previous section.

If we allow the permutation of all the indeterminates, we can generalize the results in Section 4 to the Littlewood-Richardson rule.

\begin{theorem}\label{thm:lr}
\textit{Let $\mu$ and $\nu$ be partitions and let $\bm{s}\in T(\mu,\mathbb{C})$ and $\bm{t}\in T(\nu,\mathbb{C})$. Suppose that
the real parts of all $t_{ij}$ , $s_{kl}$ are greater than 1. For a skew partition $\lambda/\mu$ such that $|\lambda/\mu|=|\nu|$, we fix a filling $\bm{u}_{\lambda}(\bm{s},\bm{t})$ of $D(\lambda)$ with $\{s_{ij} \mid (i,j)\in D(\mu)\}\cup\{t_{kl} \mid (k,l)\in D(\nu)\}$ (as indeterminates). Then, the following holds:
\[
\sum_{sym}\zeta_{\mu}(\bm{s})\zeta_{\nu}(\bm{t})=\sum_{sym}\sum_{\lambda}c_{\mu\nu}^{\lambda}\zeta_{\lambda}(\bm{u}_{\lambda}(\bm{s},\bm{t})).
\]
Here $\displaystyle\sum_{sym}$ means the summation over the permutation of $\{s_{ij} \mid (i,j)\in D(\mu)\}\cup\{t_{kl} \mid (k,l)\in D(\nu)\}$ 
as indeterminates.
}

\end{theorem}

\begin{proof}
Let $N$ be an integer such that $N\geq l(\mu)$ and $N\geq l(\nu)$.
Suppose that $L\otimes M$ corresponds to $T$ under the isomorphism $\displaystyle\mathcal{B}_{\mu}^{(N)}\otimes \mathcal{B}_{\nu}^{(N)}\to \bigoplus_{l(\lambda)\leq N}(\mathcal{B}_{\lambda}^{(N)})^{\oplus c_{\mu\nu}^{\lambda}}$. Since $\wt(L\otimes M)=\wt(T)$, we have
\[
\sum_{sym}\frac{1}{L^{\bm{s}}M^{\bm{t}}}=\sum_{sym}\frac{1}{T^{\bm{u}_{\lambda}(\bm{s},\bm{t})}}.
\]
Therefore we have
\begin{align*}
\sum_{sym}\zeta_{\mu}^{(N)}(\bm{s})\zeta_{\nu}^{(N)}(\bm{t})&=\sum_{sym}\sum_{L\otimes N\in \mathcal{B}_{\mu}^{(N)}\otimes \mathcal{B}_{\nu}^{(N)}}\frac{1}{L^{\bm{s}}M^{\bm{t}}}\\
&=\sum_{sym}\sum_{l(\lambda)\leq N}\sum_{T\in\mathcal{B}_{\lambda}^{(N)}}c_{\mu\nu}^{\lambda}\frac{1}{T^{\bm{u}_{\lambda}(\bm{s},\bm{t})}}& \text{($\because$ Theorem \ref{thm:lrc})}\\
&=\sum_{sym}\sum_{l(\lambda)\leq N}c_{\mu\nu}^{\lambda}\zeta_{\lambda}^{(N)}(\bm{u}_{\lambda}(\bm{s},\bm{t})).
\end{align*}

By Proposition \ref{thm:c1}, the sums $\displaystyle \sum_{L\in SSYT(\mu)}\frac{1}{L^{\bm{s}}},\sum_{M\in SSYT(\nu)}\frac{1}{M^{\bm{t}}}$ and $\displaystyle\sum_{T\in SSYT(\lambda)}\frac{1}{T^{\bm{u}_K(\bm{s},\bm{t})}}$ are absolutely convergent. Taking the limit $N\to \infty$, we get the desired equality by Remark \ref{thm:ww}.
\end{proof}

\section{Acknowledgment}
The author would like to thank his supervisor Syu Kato for his helpful advice and continuous encouragement.

\end{document}